\documentclass[11pt]{amsart}
\usepackage{amsmath,amsthm,amssymb,latexsym,graphicx,verbatim,enumerate,ifpdf,mdwlist,xspace, bm}
\usepackage{mathabx}
\ifpdf
\usepackage{hyperref}
\else
\usepackage[hypertex]{hyperref}
\fi
\usepackage{color}
\usepackage{amssymb, amsmath, amsthm}
\usepackage{graphicx}
\usepackage{cancel}
\usepackage{float}

\usepackage{verbatim}

\newcommand{\Ceers}{\textbf{Ceers}\xspace}

\providecommand{\Z}{\mathbb{Z}}

\providecommand{\sugs}{{$*$-universal groups}\xspace}
\providecommand{\sug}{{$*$-universal group}\xspace}

\newtheorem{question}{Question}

\newtheorem{convention}{Convention}
\newtheorem{thm}{Theorem}[section]
\newtheorem{definition}[thm]{Definition}

\newtheorem{remark}[thm]{Remark}
\newtheorem{lemma}[thm]{Lemma}

\newtheorem{prop}[thm]{Proposition}

\newtheorem{cory}[thm]{Corollary}

\newenvironment{defn}{\begin{definition} \rm}{ \end{definition}}
\newtheorem{lem}[thm]{Lemma}

\renewcommand{\phi}{\varphi}
\renewcommand{\setminus}{\smallsetminus}

\usepackage{color, soul}
\usepackage[normalem]{ulem}

\title[Algorithmically finite, universal, and $*$-universal groups]{Algorithmically finite, universal, and $*$-universal groups}

\author[U.~Andrews]{Uri Andrews}
\address{Department of Mathematics\\
University of Wisconsin\\
Madison, WI 53706-1388\\
USA}
\email{\href{mailto:andrews@math.wisc.edu}{andrews@math.wisc.edu}}

\author[M.C.~Ho]{Meng-Che ``Turbo'' Ho}
\address{Department of  Mathematics, California State University, Northridge}
\email{\href{mailto:turbo.ho@csun.edu}{turbo.ho@csun.edu}}

\keywords{}

\thanks{The second author acknowledges support from the National Science Foundation under Grant No.\ \mbox{DMS-2054558}. The authors would like to thank the CSU Desert Studies Center for supporting this project.}

\begin{document}

\begin{abstract}
The study of the word problems of groups dates back to Dehn in 1911, and has been a central topic of study in both group theory and computability theory. As most naturally occurring presentations of groups are recursive, their word problems can be thought of as a computably enumerable equivalence relation (ceer). In this paper, we study the word problem of groups in the framework of ceer degrees, introducing a new metric with which to study word problems. This metric is more refined than the classical context of Turing degrees. 

Classically, every Turing degree is realized as the word problem of some c.e.\ group, but this is not true for ceer degrees. This motivates us to look at the classical constructions and show that there is a group whose word problem is not universal, but becomes universal after taking any nontrivial free product, which we call $*$-universal. This shows that existing constructions of the Higman embedding theorem do not preserve ceer degrees. We also study the index set of various classes of groups defined by their properties as a ceer: groups whose word problems are dark (equivalently, algorithmically finite as defined by Miasnikov and Osin), universal, and \sugs. 

\end{abstract}

\maketitle

\section{Introduction}

The study of algorithmic properties of groups originated with Dehn \cite{De11} in 1911. Dehn introduced the notion of the word problem of a group, and asked if a recursively presented (or even finitely presented group) might have a non-solvable word problem. Novikov \cite{No55} and Boone \cite{Bo57} gave examples of groups which are finitely presented yet have non-solvable word problem.

A modern solution to Dehn's question uses the Higman embedding theorem \cite{Hi61}. Higman embedding states that any recursively presented group can be effectively embedded into a finitely presented group. Using this, if you want a finitely presented group with non-solvable word problem, you first construct a recursively presented group $G$ with non-solvable word problem then embed $G$ into a finitely presented $H$. Since the word problem of $G$ reduces to that of $H$, the word problem of $H$ is non-solvable as well. 

For any fixed non-computable c.e.\ set $A$, one can let $G_A$ be the group with presentation $\langle \{g_i : i\in\omega\} \mid \{g_i^2=1: i\in \omega\}\cup \{g_i=1 : i\in A\}\rangle$, and $H_A$ be the finitely presented group given by the Higman embedding theorem. Clapham \cite{Cl64} later noted that the Higman embedding can be performed so that the Turing degree of the word problem of $H_A$ is the same as the Turing degree of the word problem of $G_A$. Putting this together, the degrees of word problems of finitely presented groups are exactly the c.e.\ Turing degrees (see also \cite{Fr62, Bo66, Bo67}). This seemed to be a somewhat complete answer to the question of the complexity of word problems, but we argue that this is an incomplete picture.

We introduce another metric with which to study word problems: The degree structure of Ceers. Given a recursively presented group with computable generating set $X$, we define the word problem $W_G$ of $G$ to be an equivalence relation of the set of words in $X$ given by $w \mathrel{W_G} v$ if and only if $w=_G v$. That is, two words are equivalent if they represent the same element of the group. By fixing a bijection between the words in $X$ and $\omega$, we can consider $W_G$ to be an equivalence relation on $\omega$. Further, observe that $W_G$ is a c.e.\ set. So, $W_G$ is a c.e.\ equivalence relation (ceer).

\begin{defn}
	For two equvalence relations $E$ and $R$ on $\omega$, we say that $E$ is computably reducible to $R$ (written $E\leq R$) if there is a computable function $f:\omega \rightarrow \omega$ so that $\forall n,m \left( n\mathrel{E}m \leftrightarrow f(n) \mathrel{R} f(m) \right)$

	Let \Ceers be the degree structure given by the collection of c.e.\ equivalence relations under the partial order $\leq$.
\end{defn}

We pose that the structure \Ceers is the right setting to measure complexity of word problems $W_G$ of recursively presented groups $G$, rather than the structure of Turing degrees. Here, the narrative is far subtler and more interesting than in the Turing degree setting. For one, not every degree in \Ceers contains the word problem of a group \cite{GaoGerdes}. Further, there are degrees that contain the word problem of groups, but not of finitely generated groups \cite{De23}\footnote{\cite{De23} defined the class of hyperdark ceer degrees and showed that they cannot be realized as the word problem of a finitely generated algebra of finite type. They then proceed to construct a recursively presented semigroup with a hyperdark word problem. Using the same idea, one can also construct a recursively presented group with a hyperdark word problm.}.

Further, there are natural subclasses of \Ceers that correspond to interesting properties of groups.

\begin{defn}[Miasnikov--Osin \cite{MO}]\label{def:alg-fin}
A (finitely generated) group $G$ is \emph{algorithmically finite} if there is no infinite c.e.\ set of words whose natural image in $G$ consists of pairwise distinct elements. 
\end{defn} 

This same notion (when $G$ is infinite) was studied in the setting of \Ceers:

\begin{defn}[Andrews--Sorbi \cite{An19}]\label{def:dark}
	A ceer $E$ is \emph{dark} if it has infinitely many classes, yet whenever $W$ is an infinite c.e.\ set, there are $a,b\in W$ so that $a\mathrel{E}b$.
\end{defn}

Miasnikov and Osin \cite{MO} and Khoussainov and Miasnikov \cite{Kh14} gave constructions of finitely generated recursively presented groups which are algorithmically finite (i.e., have dark word problem).

In this paper, we show that there is a rich collection of algorithmically finite groups. In fact, in Section \ref{sec:dark}, we show that the set of recursive presentations of finitely generated groups which are algorithmically finite is a $\Pi^0_3$-complete set.

Next, we change gears to understanding the complexity of having word problem in a certain degree. Boone and Rogers \cite{Bo66a} showed that the set of finite presentations of groups which have solvable word problem is a $\Sigma^0_3$-complete set. As a corollary of this result, they show that there is no universal algorithm for solving the word problem of every finitely presented group with solvable word problem. A modern approach would note that the $\Sigma^0_3$-completeness also follows from the construction sending a c.e.\ set $A$ to $H_A$ as above. We note that this same construction shows that the set of finitely presented groups whose word problems are Turing equivalent to $0'$ is a $\Sigma^0_4$-complete set. 

There is a single largest degree in \Ceers, which we call the \emph{universal degree}. We say that a group $G$ has universal word problem if $W_G$ is in the universal ceer degree. Analogously to considering finitely presented groups of Turing degree $0'$, we consider the collection of finite presentations of groups $G$ so that $W_G$ is a universal ceer. We show that this is a $\Sigma^0_3$-complete set in Section \ref{sec:universal}.

We prove the $\Sigma^0_3$-hardness of this set using Clapham's result that we can embed a recursively presented group $G$ into a finitely presented group $H$ with the same Turing degree. In other words, we give a reduction of $(\Sigma^0_3,\Pi^0_3)$ to (universal word problem, word problem of lower Turing degree than universal). Given that a primary thesis in this paper is that the structure \Ceers, being more refined than the structure of Turing degrees, is the right place to study word problems, having to rely on Turing degree differences here is unsatisfying.

This causes us to ask if some form of Clapham's theorem could be true in ceers. Perhaps if $G$ is a recursively presented group which is non-universal, then $G$ effectively embeds in a finitely presented group $H$ whose word problem is also non-universal. While we do not fully resolve this question, we show that no construction similar to the Higman embedding construction can possibly work. In particular, Higman embedding is based on using free products and HNN extensions as introduced by Higman, Neuman, and Neuman \cite{Hi49}. Ostensibly, free products should be the simpler of these two constructions.
In Section \ref{sec:sugs}, we construct a recursively presented group $G$ which has non-universal word problem, yet any non-trivial free product of $G$ has universal word problem. We call such groups \sugs.

In Section \ref{sec:direct-prod}, we then show that although the free product of non-universal groups may be universal, and indeed \sugs exist, cross-products do not suffer this fate. If $G$ and $H$ have non-universal word problem, then $G\times H$ has non-universal word problem. Even infinite sums have this property: If $G_i$ is a uniform sequence of groups with non-universal word problems, then $\oplus_i G_i$ still has non-universal word problem.

We use this to show that even among the groups whose word problems have Turing degree $0'$, the property of universality is $\Sigma^0_3$-complete.

Finally, in Section \ref{sec:index-of-sugs}, we show that the collection of \sugs is d-$\Sigma^0_3$-complete.


%
%
%

\subsubsection*{Remark on finite generation:}
	
	When we can, we prove results about finitely presented groups. For finitely presented, or even finitely generated, groups, the ceer degree of the word problem does not depend on the presentation. 
	
	\begin{prop}
		Suppose $G = \langle S\mid R \rangle = \langle S'\mid R'\rangle$ such that $S$ and $S'$ are both finite. Then the word problem of $G$ with respect to $S$ and $S'$ are of the same degree (ceer degree if both $R$ and $R'$ are c.e.). 
	\end{prop}
	\begin{proof}
		Fix (non-uniformly) a representative for each $s' \in S'$ as a word in $S$. This induces a computable map from $(S')^*\to S^*$, which is a reduction from the word problem of $G$ with respect to $S$ to the word problem of $G$ with respect to $S'$.
	\end{proof}
	
	Note that for infinitely generated groups, the ceer degree does depend on the presentation. For instance, for every co-infinite c.e.\ set $A$, the recursive presentation of the group $G_A=\langle g_i \mid \{g_i^2=1\mid i\in \omega\}\cup \{g_i=1\mid i\in A\}\rangle$ gives a group isomorphic to $\oplus_i \Z/2\Z$, yet the ceer degree of $G_A$, and even the Turing degree of $G_A$, depend on the set $A$.

\subsubsection*{Remark on indexing and index sets:}
	There is a computable enumeration of all ceers $(E_i)_{i\in \omega}$, a computable enumeration of all finite presentations of groups, and a computable enumeration of all recursive presentations of groups. Further, these enumerations have universality properties so that, given a computable enumeration of an equivalence relation, one can effectively find an index in our enumeration of an equivalence relation in the same ceer degree, and in fact the same equivalence relation on $\omega$. Similarly for the enumerations of group presentations. We implicitly use these enumerations when we discuss the index set of ceers with some property or the index set of recursive presentations of groups with some property.

\section{The index set of darkness}\label{sec:dark}

\subsection{Algebra preliminaries}

Following \cite[\S 4.1]{Kh14}, we first recall some definitions and basic facts about the polynomial ring as an algebra and the Golod--Shafarevich theorem.

Let $(R, +, \cdot)$ be a ring and $K$ a field. recall that $R$ is called an \emph{(associative) algebra over $K$} if there is a \emph{scalar multiplication} function $*: K\times R \to R$ so that $(R, +, *)$ is a vector space over $K$ and $k*(r\cdot s) = (k*r)\cdot s = r\cdot (k*s)$. We will abuse notation and write both the ring multiplication and scalar multiplication as $\cdot$. 

Let $p$ be a prime. We will consider the \emph{non-commutative} polynomial ring $F = (\Z/p\Z)[x,y]$ as an algebra over the field $\Z/p\Z$. A polynomial is \emph{homogeneous} if every term in it has the same degree. Every polynomial can be written as a sum of homogeneous polynomials called its \emph{homogeneous components}. Let $F_k$ be the subspace consisting of all homogeneous polynomials of degree $k$ and 0, we have that $F$ is a direct sum (as vector spaces) of $F_k$, namely $F = \bigoplus\limits_{k\in \omega} F_k$. 

A subset of $F$ is called \emph{homogeneous} if every element of it is homogeneous, and an ideal of $F$ is called \emph{homogeneous} if it can be generated by a homogeneous set. Let $I$ be a homogeneous ideal, then a polynomial is in $I$ if each of its homogeneous components is in $I$. As a consequence, we have $F/I = \bigoplus (F_k+I)/I$. Also note that $F_k$ and hence $(F_k+I)/I$ are finite, so $F/I$ is computable.

Golod and Shararevich \cite{Go64} gave various conditions under which $F/I$ has infinite dimension. We will need the following variation. The condition on $n_k$ will be referred to as the \emph{Golod--Shafarevich condition} for the rest of this section. 

\begin{thm}[Golod--Shafarevich theorem {\cite[Theorem 4.3]{Kh14}}]
Let $I$ be a homogeneous ideal generated by a homogeneous set $H$, and $n_k$ be the number of homogeneous polynomials of degree $k$ in $H$. Let $0 < \epsilon \le 1$. If $n_0 = n_1 = 0$ and for every $k \ge 2$ we have 
$$n_k \le \epsilon^2 (2-2\epsilon)^{k-2},$$ 
then the dimension (as a $\Z/p\Z$ vector space) of $A = F/I$ is infinite. 
\end{thm}

\subsection{Rings} 

In \cite{Kh14}, various algorithmically finite algebraic structures are constructed. The definition is the same as saying that the word problem of said (infinite) structure is \emph{dark}, which we define here (see also Definition \ref{def:alg-fin} and \ref{def:dark}).

\begin{definition}
A ceer $E$ is \emph{light} if there is some infinite c.e.\ set $S$ so that $i\cancel{\mathrel{E}}j$ for any $i\neq j$ from $S$.

A ceer $E$ is \emph{dark} if it has infinitely many classes and is not light. 

We say a computable algebraic structure (for instance, a ring or a group) is \emph{dark} if its word problem is dark as a ceer. 
\end{definition}

In the following theorem, we follow the strategy from Khoussainov--Miasnikov \cite{Kh14} where they construct a residually finite group with a dark word problem.

\begin{thm}
The index set of finitely generated computable ring presentations whose word problem is dark is $\Pi^0_3$-complete.
\end{thm}

\begin{proof}
It is well-known that every c.e.\ ring presentation is isomorphic to a computable presentation via the \emph{padding trick}. Indeed, if $\langle S \mid r_1, r_2, \cdots \rangle$ is a c.e.\ ring presentation where $r_i$ gets enumerated at stage $s_i$, then $\langle S \mid r_1+s_11-s_11, r_2+s_21-s_21, \cdots\rangle$ is a computable presentation that defines an isomorphic ring. Thus, it suffices to construct c.e.\ presentations in this proof.

The index set of finitely generated computable ring presentations whose word problem is dark is $\Pi^0_3$ as it can be described by $\forall m, (W_m\text{ is infinite}) \implies (\exists i,j\in W_m, iEj$). Note that since $iEj$ is $\Sigma^0_1$,  
$\exists i,j\in W_m, iEj$ is $\Sigma^0_1$. The complexity lies in determining if $W_m$ is infinite, which is $\Pi^0_2$.

Given a $\Pi^0_3$ set $S$, we can effectively fix a sequence of c.e.\ sets $U_i$ such that $i\in S$ iff for every $n$, $U_i^{[n]}$, the $n$-th column of $U_i$, is finite. We will build a c.e.\ ring presentation $A_i = (\Z/p\Z)[x,y]/I_i$ such that its word problem is dark iff $i\in S$. Note that $(\Z/p\Z)[x,y]$ has a c.e.\ presentation, so it suffices to find a c.e.\ generating set $H_i$ of $I_i$. The $H_i$ we build will only contain homogeneous polynomials and satisfy the Golod--Shafarevich condition. We will start with $H_i = \emptyset$. In order to not overburden notation with subscripts, from here on we will suppress the subscript $i$, referring instead to $H$ as we describe a uniform construction producing an $H_i$ for each $i$. Below, when we refer to $H_s$, this refers to the part of the set $H$ which is enumerated by stage $s$.

We have the following list of requirements:

\begin{itemize}
	\item[$L_n$:] If $U^{[n]}$ is infinite, construct an infinite c.e.\ subset $T_n \subseteq (\Z/p\Z)[x,y]$ such that every two words in $T_n$ are not equal in $A$. 
	\item[$D_m$:] Ensure there are two words $u,v$ in $W_m$ which are equal in $A$ whenever $W_m$ is infinite.
\end{itemize} 


Of course, the $D_m$-requirements contradict the $L_n$-requirements if some $U^{[n]}$ is infinite. If $n$ is least so that $U^{[n]}$ is infinite, we will ensure that the $L_n$-requirement makes $A$ light. On the other hand, if every $U^{[n]}$ is finite, then we will ensure that all $D_m$-requirements succed, ensuring that $A$ will be an infinite algebra whose word problem is dark. We order the priority of the requirements by $L_1$, $D_1$, $L_2$, $D_2$, \dots.

For the requirements $L_n$ at stage $s$, if an element gets enumerated into $U^{[n]}$, we choose a large $k$ and find a monomial in $F_k\setminus (H_s)$ (recall that $(H_s)$ denotes the ideal generated by $H_s$) and enumerate it into $T_n$. 
Note that this is always possible if we maintained the Golod--Shafarevich condition for $H_s$. We then \emph{protect degree $k$}, i.e., require that no lower priority $D_m$-requirements add any (homogeneous) polynomial of degree $\le k$ to $H$. Note that $L_n$ will not injure any other requirements, nor does it add relations to $H$. $T_n$ consists of only monomials. If the $L_n$-strategy is reinitialized (which happens only due to the action of a higher-priory $D_m$-requirement), then we simply reset $T_n$ to be empty and the strategy starts anew.

For the requirements $D_m$ at stage $s$, let $k_s$ be the maximum of the degrees that a higher priority $L_n$-requirement protects, or $m+10$, whichever is larger. $D_m$ checks if there are two words $f,g$ in $W_{m,s}$ such that $f-g = 0$ in $A_s / (F_{k_s+1})$. If not, it simply waits. If there is, $D_m$ acts by adding each homogeneous component of $f-g$ into $H$. Each of the relations added will have degree $> k_s$.  Note that this respects the higher priority requirements. Once a $D_m$ has acted, it will not act again. $D_m$ will never get injured, although it may injure lower priority $L_m$. 

\begin{lem}
The resulting algebra $A$ satisfies the Golod--Shafarevich condition, thus is infinite. 
\end{lem}

\begin{proof}
We choose $\epsilon = 1/4$. Note that only $D_m$ adds relations to $H$, and every $D_m$ will add at most 1 relation of each degree $k \ge m+10$. Thus, the number $n_k$ of polynomials of degree $k$ in $H$ is at most $k-10$. So we have 
$$n_k \le k-10 \le (1/4)^2(3/2)^{k-2}$$
satisfying the Golod--Shafarevich condition. 
\end{proof}

\begin{lem}\label{dark-ring}
If $i\in S$, then $A = A_i$ is dark.
\end{lem}
\begin{proof}
If $i\in S$, then every $U^{[n]}$ is finite. We argue that each $D_m$-strategy succeeds. Fix $m$. There is a stage that all the higher priority $L_n$-strategies stabilize, thus $k_s$ also stabilizes. Note that $A_s / (F_{k_s})$ is a finite dimensional vector space over the finite field $\Z/p\Z$, so is finite. If $W_m$ is infinite, then $D_m$ will eventually see two words $x,y \in W_m$ such that $x-y = 0$ in $A_s / (F_{k_s})$. Thus, if $W_m$ is infinite, the $D_m$-strategy will be able to act, making sure that $W_m$ does not contain all distinct elements, and the requirement is satisfied. Since every $D_m$-requirement is satisfied, $A$ is dark.
\end{proof}

\begin{lem}
If $i\notin S$, then $A = A_i$ is light.
\end{lem}
\begin{proof}
If $i\notin S$, then there is some smallest $n$ such that $U^{[n]}$ is infinite. There is a stage such that all higher priority $D_m$-strategies have stabilized as they each act at most once, and thus the $L_n$-strategy will not be reinitialized after this stage. Then $L_n$ will build an infinite $T_n$ consisting of nontrivial monomials, one for each degree. If $f,g\in T_n$ are two monomials, then $f$ and $g$ are not in $H$, so $f-g \neq 0$ and $f$ and $g$ are distinct. Thus $T_n$ witnesses the lightness of $A$.
\end{proof}

\end{proof}

\subsection{Groups}

We are now ready to prove the $\Pi^0_3$-completeness of the set of finitely generated computable dark group presentations. We will utilize the construction in the previous subsection with a slight change. The group we construct will be the subgroup generated by $1+x$ and $1+y$ in the group of units of $A$. However, a priori, $1+x$ and $1+y$ may not be invertible in $A$. We ensure they are invertible by adding the relations $x^{10} = y^{10} = 0$ when we initialize the construction. With these relations, $1+x$ (and similarly $1+y$) is invertible with inverse $1-x+x^2-x^3+\cdots-x^9$.

We will see that we can maintain the Golod--Shafarevich condition with these two relations added. Furthermore, we recall that the ring $F$ is non-commutative. This is important as setting $x^{10} = y^{10} = 0$ does not trivialize $F_{20}$, which contains non-zero elements like $x^5y^5x^5y^5$. 

\begin{thm}
The index set of finitely generated computable group presentations whose word problem is dark is $\Pi^0_3$-complete.
\end{thm}

\begin{proof}
We will follow the proof as the ring case, with the following changes. 
\begin{enumerate}
\item When initializing the construction, before any requirements act, we let $H = \{x^{10}, y^{10}\}$ (instead of $H = \emptyset$). 
\item For every $i\in\omega$, instead of building $A_i$, we will build (a c.e.\ presentation of) $G_i$ which is the subgroup of the group of units of $A_i$, generated by $1+x$ and $1+y$. 
\item When constructing $T_n$, whenever $L_n$ acts, instead of putting $u = x^{j_1}y^{j_2}x^{j_3}\dots$ into $T_n$, it puts $v = (1+x)^{j_1}(1+y)^{j_2}(1+x)^{j_3}\dots$ into $T_n$. 
\end{enumerate}

Note in particular that the effect of (3) on other requirements is exactly the same as in the ring construction. Namely, the strategy protects the degree $k$. Only the single $L_n$-strategy is concerned with the content of the enumerated set $T_n$.
The rest of the construction is unchanged.

\begin{lem}
$G_i$ has a computable presentation.
\end{lem}
\begin{proof}
We will use $1+x$ and $1+y$ as the generators in the presentation. Their inverses are $1-x+x^2-x^3+\cdots-x^9$ and $1-y+y^2-y^3+\cdots-y^9$. Thus, for every word in $1+x$ and $1+y$, we can computably find its images in $A_i$. This allows us to computably enumerate all the relations that hold on $1+x$ and $1+y$ from the c.e.\ presentation of $A_i$. We then use the padding trick to obtain a computable presentation. 
\end{proof}

\begin{lem}
$G_i$ is infinite.
\end{lem}
\begin{proof}
We first check that $A_i$ is infinite since it still satisfies the Golod--Shafarevich condition. Starting with $H = \{x^{10},y^{10}\}$ only changes the single number $n_{10}$ in the Golod--Shafarevich condition, and it is straightforward to check that the Golod--Shafarevich inequality is still maintained. So $A_i$ is infinite. 

We claim that the image of $G_i$ generates $A_i$ as a $\mathbb{Z}/p\mathbb{Z}$-vector space. That is, $A_i$ can be obtained by taking the additive 
closure of $G_i$. Indeed, $1$, the identity of the group, spans $F_0$; $x = (1+x)-1$ and $y = (1+y)-1$ spans $F_1$; and every degree $k$ monomial $x^{j_1}y^{j_2}x^{j_3}\dots$ is equivalent to $(1+x)^{j_1}(1+y)^{j_2}(1+x)^{j_3}\dots$ modulo $F_{k-1}$. Since $\Z/p\Z$ has characteristic $p$, if $G_i$ were to be finite, then its span would also be finite, but $A_i$ is infinite. 
\end{proof}

\begin{lem}
If $i\in S$, then $G_i$ is dark.
\end{lem}
\begin{proof}
Any c.e.\ subset of $G_i$ is also a c.e.\ subset of $A_i$, and if two elements are equal in $A_i$ then they must be equal in $G_i$. By Lemma \ref{dark-ring}, if $i\in S$ then $A_i$ is dark, and so is $G_i$.
\end{proof}

\begin{lem}
If $i\notin S$, then $G_i$ is light.
\end{lem}
\begin{proof}
The lightness of $G_i$ will be witnessed by $T_n$ where $n$ is the smallest number with $U^{[n]}$ infinite. As before, after all higher priority $D_m$-strategies stabilize, the $L_n$-strategy will act infinitely many times and $T_n$ will be infinite. Suppose towards a contradiction that two elements of $T_n$ are equal. Then we have $v = (1+x)^{j_1}(1+y)^{j_2}(1+x)^{j_3}\dots$ and $v' = (1+x)^{j'_1}(1+y)^{j'_2}(1+x)^{j'_3}\dots$ being equal. Without loss of generality, suppose the degree of $v$ is higher. Working in $A_i$, we have $v-v'= 0$. Thus, each of the homogeneous components of $v-v'$ equal zero. In particular, the homogeneous component of the highest degree, $u = x^{j_1}y^{j_2}x^{j_3}\dots = 0$. However, $L_n$ put $v = (1+x)^{j_1}(1+y)^{j_2}(1+x)^{j_3}\dots$ into $T_k$ in the group construction because it would have put $u = x^{j_1}y^{j_2}x^{j_3}\dots$ into $T_k$ in the ring construction. This only happens if $u \neq 0$ at that stage and $L_n$ will protect the degree of $u$, making $u$ nontrivial, a contradiction. Thus every pair of elements in $T_n$ are distinct, witnessing the lightness of $G_i$. 
\end{proof}

\end{proof}

\section{The index set of universality}\label{sec:universal}

We next explore the index set of finitely presented groups whose word problem is universal. We note that the property of a ceer being universal is a $\Sigma^0_3$-complete property \cite{An16}. Yet this result holds for any ceer degree which contains ceers with infinitely many classes:

\begin{thm}[\cite{An16}]
	Let $E$ be an equivalence relation with infinitely many classes, then the index set of ceers which are equivalent to $E$ is a $\Sigma^0_3$-complete set.
\end{thm}

Thus we might expect the same to hold in the setting of finite presentations of groups, yet we know that for some ceer degrees, the set of finite presentations of groups which land in that degree is empty, thus computable. We are only able to characterize this index set for the universal degree.

\begin{thm}\label{thm:universal}
	The index set of finitely presented groups whose word problem is universal is $\Sigma^0_3$-complete.
\end{thm}
\begin{proof}
	
We first describe the intuition of the proof. The first step is to construct, for any $\Sigma^0_3$ set $S$, a uniform sequence $(E^i)_{i\in \omega}$ of ceers so that $E^i$ is universal (as a ceer) if $i\in S$ and low (as a Turing degree) if $i\notin S$. We then consider a uniform procedure to embed each $E^i$ into the word problem of a finitely presented group $H^i$ and show that the dichotomy still holds for $H^i$. If $i\notin S$, then the word problem for $H^i$ cannot be universal since it is low. If $i\in S$, then the effective embedding of $E^i$ into the word problem of $H^i$ will show that the word problem for $H^i$ is universal.	

\begin{lemma}
	For any $\Sigma^0_3$ set $S$, there is a uniform sequence of ceers $E^i$ so that $E^i$ is universal if $i\in S$ and $E^i$ has low Turing degree if $i\notin S$.
\end{lemma}
\begin{proof}
	We note that this construction is similar to the one in \cite[Theorem 5.2]{An14}. 
	Fix a universal ceer $U$. We can effectively find indices for a sequence $W_i$ of c.e.\ sets so that $i\in S$ if and only if $\exists j (W_i^{[j]} \text{ is infinite})$.

	We construct $E^i$ as a uniform join $E^i = \oplus_j X^j$. We have requirements as follows:
	
	$C_k$: If $W_i^{[k]}$ is infinite, then there is some $j$ so that $X^j = U$.
	
	$L_m$: If $\phi_{m,s}^{E^i_s}(m)\downarrow$ for infinitely many $s$, then 
$\phi_m^{E^i}(m)	\downarrow$.

We order their priority by $C_0, L_0, C_1, L_1, \dots$.

We want to construct $E^i$ so that if no column of $W_i$ is infinite, then we satisfy every $L_m$-requirement, and if $k$ is the least such that $W_i^{[k]}$ is infinite, then $C_k$ is satisfied. Thus, $E^i$ is universal in the $\Sigma^0_3$-outcome,  and we ensure that $E^i$ has low Turing degree in the $\Pi^0_3$-outcome. 


To satisfy $C_k$, if a new number is enumerated into $S^{[k]}$, we act by:
\begin{itemize}
\item If $C_k$ is not yet initialized, we initialize it by choosing a new parameter $j$ so that the set $X^j$ is not restrained by any higher priority $L_m$, and let $X^j_s = U_s$. 
\item If $C_k$ is already initialized, we let $X^j = U_s$, i.e., make $X^j$ catch up with the current $U_s$. 
\end{itemize}

To satisfy $L_m$, whenever $\phi_{m,s}^{E^i_s}(m)\downarrow$ holds, we place a restraint on the use of this computation.

Whenever we act (including placing restraint), all lower-priority strategies are reinitialized. 


The construction is put together as a standard finite injury argument. In the $\Sigma^0_3$-outcome, there is a least $k$ such that $W_i^{[k]}$ is infinite. Choose a stage such that every previous (finite) column $W_i^{[\ell]}$ and every higher priority $L_m$ has stabilized. Then at the next stage $s$ that $W_i^{[k]}$ acts, it will choose a new column $j$, which will also stabilize for the rest of the computation. Since $W_i^{[k]}$ will act infinitely often and never be injured, we will have $X^j = U$, making $E^i$ universal.

In the $\Pi^0_3$-outcome, we argue that every $L_m$ is satisfied, making $E_i$ low: Fix $L_m$ and choose a stage so that every higher priority $C_k$ and $L_{m'}$ has stabilized. Then if it acts again, it will restrain its use and never get injured, satisfying the $L_m$ requirement. 
\end{proof}

Clapham showed that any group $G$ with a c.e.\ presentation can be embedded into a finitely presented group $H$  such that the word problem of $G$ and $H$ have the same Turing degree \cite{Cl67}, see also \cite[Chapter IV.7]{Ly01}. The construction is explicit and one can check that it is uniform (in the index of the c.e.\ presentation) and effective. We note that the fact that this is effective uses the fact that the Matiyasevich theorem is effective.

For each ceer $E^i$, we give an embedding of $E^i$ into a finitely presented group $H$. First, we construct the recursively presented group $G^i$ with generators $\{g_k\mid k\in \omega\}$ and relations $\{g_k^2 = 1 \mid k\in \omega\}\cup \{g_kg_j=g_j g_k\mid j,k\in \omega\}\cup \{g_j=g_k\mid j\mathrel{E^i} k\}$. Essentially $G^i$ is the $\mathbb Z/2\mathbb Z$-module generated by the classes of $E^i$. the map $n\mapsto g_n$ gives an embedding of $E^i$ into the word problem of $G^i$. The Turing degree of the word problem of $G^i$ is the same as the Turing degree of $E^i$. Finally, form the groups $H^i$ by employing Clapham's theorem on $G^i$. If $i\in S$, then the word problem of $H^i$ is universal as it embed the word problem of $G^i$, and thus $E^i$; and if $i\notin S$, then the word problem of $H^i$ has the same Turing degree as the word problem of $G^i$ and $E^i$, so is low and cannot be universal.
\end{proof}

We note that we used Clapham's result to get the non-universality of $H^i$ for $i\notin S$ by using the Turing degree. This is somewhat unsatisfying, since it does not seem to support the thesis that the ceer degrees give us a refined setting to explore complexity of word problems. We will resolve this complaint below in Corollary \ref{cor:universalityIsSigma3Complete-fp} showing that even within the Turing degree of $0'$, universality is still $\Sigma^0_3$-complete.

In the next section, we explore to what extent we could hope to prove this theorem using a ceer-version of Clapham's result.

\section{\sugs}\label{sec:sugs}

By work from Della Rose, San Mauro, and Sorbi \cite{De23}, we know that there are recursively presented groups whose ceer degree contains no word problem of a finitely generated group.\footnote{See footnote 1 in the introduction.} So, we know that there is no version of Clapham's result which preserves ceer degree. What would be needed in the previous section is a version of Clapham's result simply preserving non-universality. We explore that notion here.

Since Higman embedding, including Clapham's version, is built upon the operations of free product and HNN-extension, we explore whether these preserve non-universality. 

\begin{defn}
	A recursively presentable group $G$ is a \emph{\sug} if the word problem of $G$ is not a universal ceer, yet whenever $H$ is a non-trivial group, the word problem of the free product $G*H$ is a universal ceer. 
\end{defn}

Note that applying Higman embedding (even ala Clapham) to a \sug will produce a universal group, since some of the steps require taking free products.

Next we see that the quantification over all non-trivial groups $H$ is not necessary, and we can instead consider only $G\ast \Z/2\Z$.

\begin{lem}
Let $G, H$ be groups and $g_0,\dots,g_n\in G$, $1_H\neq h\in H$. Write $\Z/2\Z = \langle a \rangle$. Then $g_0ag_1a\dots ag_n = 1$ in $G*(\Z/2\Z)$ iff $g_0h^{-1}g_1h\dots h^{(-1)^n}g_n = 1$ in $G*H$. 
\end{lem}

\begin{proof}
We induct on $n$. When $n = 0$ the statement is clear, when $n = 1$ both $g_0ag_1$ and $g_0h^{-1}g_1$ are never 1, and when $n = 2$ we have $g_0ag_1ag_2 = 1$ iff $g_1 = g_0g_2 = 1$ iff $g_0h^{-1}g_1hg_2 = 1$. 

Suppose $n \ge 3$ and $g_0ag_1a\dots ag_n = 1$. Then there is some $1 \le i\le n-1$ so that $g_i = 1$. Thus, we have $1 = g_0ag_1a\dots ag_{i-1}ag_iag_{i+1}a\dots  ag_n = g_0ag_1a\dots ag_{i-1}g_{i+1}a\dots  ag_n$. By induction hypothesis, this implies that $$g_0h^{-1}g_1h\dots h^{(-1)^{(i-1)}}g_{i-1}g_{i+1} h^{(-1)^{i}}\dots  h^{(-1)^{(n-2)}}g_n = 1.$$
On the other hand, we have $g_i = 1$, so 
\begin{align*}
&g_0h^{-1}g_1h\dots h^{(-1)^{(i-1)}}g_{i-1}h^{(-1)^i}g_i h^{(-1)^{(i+1)}}g_{i+1}h^{(-1)^{(i+2)}}\dots h^{(-1)^n}g_n\\ 
= &g_0h^{-1}g_1h\dots h^{(-1)^{(i-1)}}g_{i-1}g_{i+1}h^{(-1)^{(i+2)}}\dots h^{(-1)^n}g_n\\
= & 1.
\end{align*}
The reverse direction is similar. 
\end{proof}

\begin{cory}\label{cor:CheckingZ2isenough}
A group $G$ is a \sug iff the word problem of $G$ is not a universal ceer but the word problem of $G*(\Z/2\Z)$ is a universal ceer. 
\end{cory}

\begin{proof}
The only if direction is clear. For the if direction, supposing the word problem of $G*\Z/2\Z$ is universal, it suffices to show that the word problem $G*H$ is universal for every nontrivial $H$. Given a nontrivial $H$, we construct a reduction from the word problem of $G*\Z/2\Z$ to $G*H$. 

Fix (non-uniformly) a nontrivial element $h\in H$. Define the reduction $f$ from the word problem of $G*\Z/2\Z$ to the word problem of $G*H$ by $f(g_0ag_1a\dots ag_n) = g_0h^{-1}g_1h\dots h^{(-1)^n}g_n$. Suppose $u = g_0ag_1a\dots ag_n$ and $v = g'_0ag'_1a\dots ag'_{n'}$, so we have $f(u) = g_0h^{-1}g_1h\dots h^{(-1)^n}g_n$ and $f(v) = g'_0h^{-1}g'_1h\dots h^{(-1)^n}g'_{n'}$. We observe that $u^{-1}v = g^{-1}_n a \dots a g^{-1}_1 a g^{-1}_0g'_0ag'_1a\dots ag'_{n'}$ and $$(f(u))^{-1}f(v) = g^{-1}_n h^{(-1)^{(n+1)}}\dots h^{-1} g^{-1}_1 h g^{-1}_0 g'_0h^{-1}g'_1h\dots h^{(-1)^n}g'_{n'}.$$ 

Since the signs of the powers of $h$ in $(f(u))^{-1}f(v)$ alternate, we may apply the previous lemma to get $u = v$ iff $u^{-1}v = 1$ iff $(f(u))^{-1}f(v) = 1$ iff $f(u) = f(v)$, so $f$ is a reduction. 
\end{proof}

Finally, we give a direct construction showing the existence of \sugs.

\begin{thm}\label{thm:exists-sugs}
	There exists \sugs.
\end{thm}
\begin{proof}
	We give a direct construction of a group $G$. $G$ will be an abelian group with generators $\{x_i\mid i\in \omega\}$. Throughout the construction, we will add 4 types of relations, each of the form $x_j = w(x_0,\ldots, x_{j-1})$, to the presentation of our group. These are:
\begin{enumerate}
\item $x_j = 1$.
\item $x_j = x_i$ for some $i < j$. 
\item $x_j = x_i^{-1}$ for some $i < j$.
 \item $\displaystyle \prod_{k\in S} x_k = 1$ for some $i < j$, which is equivalent to $\displaystyle x_{k'} = \left( \prod_{k\in S\setminus\{k'\}} x_k\right)^{-1}$, where $S$ is a subset of natural numbers and $k' = \max S$. (In fact, $S$ is always either all even or all odd indices of a level $j$, explained below.)
\end{enumerate}
	
At any stage, whenever we consider a word, we always reduce it using the relations already enumerated up to that stage, by replacing the left-hand side of the relation by the right-hand side. 
	
From the previous Corollary, it suffices to make $G* \Z/2\Z$ have universal word problem. Let $a$ be the non-identity element of $\Z/2\Z$. We will give a sequence of words $(v_i)_{i\in \omega}$ in the letters $\{x_i\mid i\in \omega\} \cup \{a\}$. We will ensure that $v_i \mathrel{W_{G* \Z/2\Z} } v_j$ if and only if $i \mathrel{U} j$ for a fixed universal ceer $U$. 
	
	We fix the words $$v_i:= \prod\limits_{k= 10^i}^{10^{i+1}-1} ax_k.$$ Note the role of $a$ is in separating the elements $x_k$ of $G$ so they do not combine in the free product $G* \Z/2\Z$. 
	
	At the beginning of the construction, for every $j$, we add the following relations to $G$:
$$\prod_{\substack{k= 10^j \\ k\text{: even}}}^{10^{j+1}-2} x_k = 1$$
$$\prod_{\substack{k= 10^j+1 \\ k\text{: odd}}}^{10^{j+1}-1} x_k = 1$$
Note that after reducing $v_i$ using these relations, the $x_{2\ell}$ with the largest even index becomes the inverse of the product of all the previous $x_{2\ell'}$ and similarly for the last $x_{2\ell+1}$.

For each $k\in [10^j,10^{j+1})$ aside from the last two elements (which we consider already determined by the two added relations in $G$), we say $x_k$ is a \emph{level $j$ generator}. We may at a later stage say that $x_k$ is no longer a level $j$ generator. This happens in 2 possible ways: 
	\begin{itemize}
		\item We already collapsed it to being equivalent to some level $i$ generator for $i<j$ or to being equivalent to $1$.
		\item We have caused $x_k$ to cancel in the word $v_j$ with its neighbor. As such, it will contribute nothing to the word $v_j$ anymore, and we will say that it is no longer level $j$. Rather, we will say that it is \emph{free}.  
	\end{itemize} 
	
At any stage of the construction, we say $x_i$ and $x_j$ are \emph{consecutive} if in the current form of $v_i$ (after reduction using the relations already enumerated up to this stage), they are separated by only an $a$. For instance, $x_{11}$ and $x_{12}$ are initially consecutive in $v_1 = ax_{10}ax_{11}ax_{12}\cdots$. But if the relations $x_{13} = x_{15} = 1$ and $x_{12} = x_{14}^{-1}$ (making $x_{12}$ free)  get enumerated, then we have the current $v_1$ becomes $ax_{10}ax_{11}ax_{16}\cdots$, so $x_{11}$ and $x_{16}$ are now consecutive. The same applies to consecutive even or odd $x_i$, for instance, $x_{10}$ and $x_{16}$ become consecutive even generators in the example. 
	
	At any stage $s$, if we see $i<j$ become $U$-equivalent, then we take some consecutive $10^{i+1}-10^i$ generators that are currently level $j$ generators. Note that we verify in Lemma \ref{StillActiveStuff} that we will have enough current level $j$ generators. We will collapse these to being equivalent to the level $i$ generators and all the other level $j$ generators to being equivalent to $1$. This will ensure directly that $v_j = v_i$ in $G*\Z/2\Z$.
	
	We also construct an auxiliary ceer $X$ and have the requirements:
	$$R_e: \phi_e \text{ is not a reduction of $X$ to $W_G$}.$$
	
	This will ensure that $W_G$ is not a universal ceer, though we ensure that $i\mapsto v_i$ is a reduction of $U$ to $W_{G*\Z/2\Z}$, so that $W_{G*\Z/2\Z}$ is a universal ceer.
	
	In order to satisfy requirement $R_e$, we act as follows:

	Pick two new numbers $a,b$ and wait for $\phi_e(a)$ and $\phi_e(b)$ to converge. Then consider the word $w = \phi_e(a)\cdot \phi_e(b)^{-1}$. Let $G_s$ be the currently built $G$, and we observe that $G_s$ is abelian and finitely presented so has computable word problem. 
	We use the currently placed relators on $G_s$ to simplify $w$ as much as possible. In particular, for each $j$, we do not have $x_{10^{j+1}-1}$ or $x_{10^{j+1}-2}$ appearing in $w$.
	
	Our goal is to ensure that $a X b$ if and only if $w\neq 1$ in $G$. Let $K$ be greatest so that there are letters in $w$ at level $K$. We proceed based on the following cases:
	
	Case 0: If $w = 1$, we do not do anything (so $\neg aXb$) and we are done with this requirement.
	
	Case 1: There are free letters appearing in $w$. We will verify in Lemma \ref{Case0works} below that this ensures $w \neq 1$, so we simply cause $a X b$ and we are done with this requirement.
	
	Case 2: $K\leq e$, we act under the assumption that there will be no future collapse in $U\cap [0,K]^2$. That is, we look at the subgroup of $G$ generated by $\{x_i\mid i<10^{K+1}\}$. Since we have $w \neq 1$, we will assume that there is no future collapse among these letters, so we collapse $a X b$.
	
	The idea is that the subgroup generated by $\{x_i\mid i<10^{K+1}\}$ will only change if a higher priority requirement acts or if $U$ changes on $[0,e]^2$, either of which will happen only finitely often, so we are willing to rely on this and re-start the requirement (with a new choice of $a$ and $b$) if there is such a change.
	
	Case 3: $K> e$. We let $w_K=\prod_{k \in [10^j,10^{j+1})} x_k^{e_k}$ be the subword of $w$ comprising the level $K$ generators and consider further cases. 
	
	
	
	Case 3a: There are consecutive even letters: $x_{2n}$ and $x_{2n'}$ so that $e_{2n} \neq e_{2n'}$.
	
	We write $v_i = \alpha \cdot x_{2n} x_\ell^a x_{2n'} x_m^a \cdot \beta$ (or $v_i = \alpha \cdot x_m^a x_{2n} x_\ell^a x_{2n'} \cdot \beta$ if $x_{2n'}$ is the level $K$ generator with the largest even index).We collapse $x_\ell = x_m = 1$ and we collapse $x_{2n} = x_{2n'}^{-1}$. Observe that $x_\ell$ and $x_m$ are now no longer level $k$ generators but are just 1, and $x_{2n}$ and $x_{2n'}$ are now free. We observe that $x_{2n}$ and $x_{2n'}$ do not cancel out in $w$ and now $w$ contains a free generator as in case 1. We collapse $a X b$ and get the same victory as in case 1.
	
	Case 3b: There are consecutive odd letters: $x_{2n+1}$ and $x_{2n'+1}$ so that $e_{2n+1} \neq e_{2n'+1}$. This is identical to Case 3a.
	
	
	Case 3c: Every level $K$ even generator appears in $w$ with the same exponent and every level $K$ odd generator appears in $w$ with the same exponent. Then we add the relators 
	$$\prod\limits_{x \text{: even level $K$ generators}}x = 1 $$
	$$\prod\limits_{x \text{: odd level $K$ generators}}x = 1. $$
		
	This amounts to removing the largest even and largest odd level $K$ generators. This also causes reduction on $v_i$, eliminating the last even term and rewriting the second-to-last term based on the relation (and similarly for the odd terms). 
	As a result, $v_i$ maintains the same form, namely, $v_i = \prod ax_k$ with the last $x_{2\ell}$ equal to the inverse of the product of all the $x_{2\ell'}$ of level $K$ and similarly for the last $x_{2\ell+1}$. Thus, after this action we have exactly 2 less level $K$ generators and we have just ensured that $w_K = 1_G$.
	
	We now reconsider which case we are in, noting that the definition of $K$ has now dropped.
	
	\subsubsection*{Verification}
	
	We first observe that we always have enough level $j$ generators that we can respond if we see $i \mathrel{U} j$ for some $i<j$.
	
	
	\begin{lem}\label{StillActiveStuff}
		At every stage, if $j$ is currently the least member of its $U$-class, then there are more than $10^j$ level $j$ generators.
	\end{lem}
	\begin{proof}
		Observe that only the requirements $R_e$ with $e< j$ can cause action removing a level $K$ generator. Each such action can remove at most 4 level $K$ generators (in case 2a or 2b). Note that even accounting for reinitialization, which can come from injury due to higher-priority requirements acting or due to $U$-collapse below $j$, these $j$ requirements can act at most $j\cdot 2^j$ times, so we have $10^{j+1}-10^j - 4\cdot j\cdot (2^j)$ remaining level $K$ generators. Observe that $10^{j+1}-10^j-4\cdot j\cdot (2^j)>10^j$ for any $j\geq 0$.
	\end{proof}

\begin{lem}
At any stage, the group $G_s$ is freely generated by the level $i$ generators for all $i$ and the free generators.
\end{lem}
\begin{proof}
Observe that all our relations are of the form $x_j = w(x_0, \dots, x_{j-1})$, and any $x_j$ appears on the left hand side of a relation at most once. 
\end{proof}

	\begin{lem}\label{Case0works}
		If $x$ is free at stage $s$, then it never appears in any new relator. In particular, if $w$ is found for a requirement $R_e$ and $w$ (after being simplified at stage $s$) contains a free generator $x$ in it, then $w\neq 1_G$.
	\end{lem}
	\begin{proof}
		Observe that in no case do we ever add a relator involving an already free generator. 
		Thus if $x$ is already free at stage $s$ then $G$ can be written as $G_F \times G_A$ where $G_F$ is the subgroup of $G$ freely generated by the generators which are free at stage $s$ and $G_A$ is generated by all those which are level $j$ generators for some $j$. 
		
		Since all future relators being added to $G$ are purely within $G_A$, this splitting persists and the projection of $w$ onto $G_F$ is already not 1, so $w\neq 1_G$.
	\end{proof}

	\begin{lem}
		For any $i$ and $j$, $i \mathrel{U} j$ if and only if $v_i \mathrel{=_{G*\Z/2\Z}} v_j$.
	\end{lem}
	\begin{proof}
		If $i U j$, then we actively collapse generators in $G$ to ensure that $v_i =_{G*\Z/2\Z} v_j$. On the other hand, observe that the only relations that involve both level $i$ and level $j$ generators for $i \neq j$ are of the form $x_n = x_m$, introduced when we respond to a $i\mathrel{U}j$ collapse. Thus, if $ i \mathrel{\cancel U} j$, then no generators of level $i$ and $j$ are related, so we have $v_i \neq_{G*\Z/2\Z}v_j$. 
	\end{proof}

	\begin{lem}
		There is no reduction from $X$ to $W_{G}$, so $W_G$ is not a universal ceer.
	\end{lem}
	\begin{proof}
		Suppose towards a contradiction that $\phi_e$ is a reduction from $X$ to $W_G$. Let $s$ be the last stage where $R_e$ is initialized. Then $a$ and $b$ are chosen to have their final values. We consider the outcome of the $R_e$ strategy. In Case 0, we have $\neg a\mathrel{X}b$ but $\phi(a) = \phi(b)$. In Case 1, Case 3a, or Case 3b, Lemma \ref{Case0works} shows that $a \mathrel{X} b$ yet $\phi_e(a)\neq_G \phi_e(b)$. In case 2, the assumption that $R_e$ is never reinitialized shows that there is no more collapse in $U$ below $e$, so we have $a \mathrel{X} b$ if and only if $w\neq_G 1$. Finally, Case 3c can happen only finitely often as it causes $K$ to drop and the strategy continues.
	\end{proof}

\end{proof}

\section{Direct Products Do Not Achieve Universality}\label{sec:direct-prod}
\subsection{Preliminary on u.e.i.\ ceers}
	In studying the degree of the universal ceer, it is often useful to consider a combinatorial characterization of this degree. 
	\begin{defn}
		A nontrivial ceer (i.e., having at least two classes) $E$ is \emph{uniformly effectively inseparable}, or \emph{u.e.i.}\ for short, if there is a computable function $p(a,b,i,j)$ such that if $ a\mathrel{\cancel E}b$ and $[a]_E \subseteq W_i$ and $[b]_E \subseteq W_j$, then $p(a,b,i,j) \notin W_i\cup W_j$. 
	\end{defn} 
	
	\begin{thm}[{\cite[Corollay 3.16]{An14}}]
		A ceer which is u.e.i.\ is universal. Consequently, a ceer $E$ is universal if and only if there is a c.e.\ subset $X$ which is $E$-closed (i.e., $y \mathrel{E} x$ with $x\in X$ implies $y\in X$) and the restriction of $E$ to $X$ is u.e.i.. 
	\end{thm}
	
	Observe that any nontrivial ceer which is a quotient of a u.e.i.\ ceer is still u.e.i.. Thus, we have the following Lemma.
	
	\begin{lem}\label{lem:quotientsOfUeis}
		Any nontrivial quotient of a u.e.i.\ ceer is universal.
	\end{lem}
	
	Note that not all non-trivial quotients of universal ceers are universal. For instance, for $X$ universal, consider $X\oplus Y$ with $Y$ non-universal. This is universal, but has $Y$ as a quotient.

\subsection{Direct products and some applications}

In contrast to the free product in the previous section, we show that the direct product of groups with non-universal word problems cannot have universal word problem.


For ceers $A,B$, we let $A\times B$ be the ceer defined by
$$\langle a_1,b_1\rangle \mathrel{A\times B} \langle a_2,b_2\rangle \leftrightarrow a_1 \mathrel{A} a_2 \wedge b_1\mathrel{B} b_2.$$

\begin{thm}\label{cross products dont get universal}
	If $A$ and $B$ are ceers which are non-universal, then $A\times B$ is non-universal.
\end{thm}

\begin{proof}
Suppose $A \times B$ were universal, and fix a u.e.i.\ $U$ and a reduction $f$ from $U$ to $A\times B$. Let $\pi_A$ and $\pi_B$ be projections sending a pair $\langle a,b\rangle$ to its two coordinates. Define $i \mathrel{A'} j$ if and only if $\pi_A \circ f(i) \mathrel{A} \pi_A\circ f(j) $, so $A'$ is a quotient of $U$. Observe that $A'$ reduces to $A$ via $\pi_A\circ f$. Similarly define $B'$. Note that $i\mathrel{U}j$ if and only if $i\mathrel{A'}j$ and $i\mathrel{B'}j$, so at least one of $A'$ or $B'$ is nontrivial. This makes either $A'$ or $B'$ be universal by Lemma \ref{lem:quotientsOfUeis}, and thus one of $A$ or $B$ is universal.
\end{proof}

Since the word problem of $G\times H$ is the product of the word problem of $G$ and the word problem of $H$, we get the same result for the group operation $\times$.

\begin{cory}
	If $G$ and $H$ are groups with non-universal word problem then $G\times H$ has non-universal word problem.
\end{cory}

This gives a way to give an example of a finitely presented group whose word problem has Turing degree $\mathbf{0'}$, yet is not universal.

\begin{cory}\label{cor:turingcompletenotcomplete}
There is a finitely presented group whose word problem is universal as a Turing degree, but not universal as a ceer. 
\end{cory}
\begin{proof}
Let $\mathbf{a}$ and $\mathbf{b}$ be c.e. degrees whose join is $\mathbf{0'}$.
Take two finitely presented groups $G$ and $H$ so that the word problems of $G$ is in the degree $\mathbf{a}$ and the word problem of $H$ is in the Turing degree $\mathbf{b}$. This is possible by Clapham \cite{Cl64}. Then $G\times H$ has word problem with Turing degree $\mathbf{0'}$, but is not universal as a ceer. 
\end{proof}

In fact for groups we get a version of Theorem \ref{cross products dont get universal} for infinitary sums:

\begin{thm}\label{thm:oplusavoidsuniversality}
	If $(G_i)_{i\in \omega}$ is a uniform sequence of recursively presented groups so that the word problems of each $G_i$ is non-universal, then the word problem of $\oplus_i G_i$ is non-universal.
\end{thm}
\begin{proof}
	Let $U$ be a u.e.i.  and suppose that $f$ is a reduction from $U$ to $W_{\oplus_i G_i}$.  For each $j$, define $E_j$ by $i \mathrel{E_j} k$ if and only if $\pi_j\circ f(i) \mathrel{W_{G_j}} \pi_j \circ f(k)$.  Then $E_j$ is a quotient of $U$.  For some $j$, we must have that $E_j$ is a non-trivial quotient as otherwise the image of $f$ is contained in one class.  But then this $E_j$ is universal and reduces to $W_{G_j}$.
\end{proof}

\begin{cory}\label{cor:oplusmaintains*universality}
	If $G$ is $*$-universal and $H_i$ are a sequence of non-universal groups, then $G\oplus \oplus_i H_i$ is $*$-universal.
\end{cory}
\begin{proof}
	By Theorem \ref{thm:oplusavoidsuniversality}, the word problem of $G\oplus \oplus_i H_i$ is non-universal. Observe that for any group $K$, the word problem of $G\ast K$ reduces to the word problem of $(G\oplus \oplus_i H_i)\ast K$, so the $*$-universality of $G$ yields the $*$-universality of $G\oplus \oplus_i H_i$.
\end{proof}

In Section \ref{sec:universal}, we showed that the index set of finitely presented groups which are universal is a $\Sigma^0_3$-complete set. Though the solution given there is not completely satisfactory, since it relied on Turing degree in the case where the constructed group is to have non-universal word problem, we now witness the $\Sigma^0_3$-completeness of universality for finitely presented groups within the Turing degree of $\mathbf{0'}$.

\begin{cory}\label{cor:universalityIsSigma3Complete-fp}
	Given a $\Sigma^0_3$ set $S$, there is a sequence of finite presentations of groups $H_i$ so that $H_i$ is universal if and only if $i\in S$. Furthermore, the Turing degree of the word problems of each $H_i$ is $0'$.
\end{cory}
\begin{proof}
	From Theorem \ref{thm:universal}, we know that given a $\Sigma^0_3$ set $S$, we can produce a sequence of finitely presentations of groups $G_i$ so that $G_i$ is universal if and only if $i\in S$. Fix $H$ to be the group from Corollary \ref{cor:turingcompletenotcomplete}. Then the sequence $G_i\times H$ is a sequence of finitely presented groups so that $G_i\times H$ is universal if and only if $i\in S$ by Theorem \ref{cross products dont get universal}. 
\end{proof}

\section{The index set of \sugs}\label{sec:index-of-sugs}

Lastly, we consider the index set of \sugs.

\begin{thm}
	The collection of recursive presentations of \sugs is a $d$-$\Sigma^0_3$-complete set.
\end{thm}
\begin{proof}
	Though at first the definition of $G$ being $*$-universal requires quantifying over possible groups $H$ and considering the universality of $G\ast H$, Corollary  \ref{cor:CheckingZ2isenough} shows that it is equivalent to the universality of $G\ast \Z/2\Z$ and the non-universality of $G$. Thus, the collection of recursive presentations of $*$-universal groups is a $d$-$\Sigma^0_3$ set.
	
	To show completeness, fix a pair $S,T$ of $\Sigma^0_3$ sets. Given a pair $(i,j)$, we produce a group $K$ which is $*$-universal if and only if $i\in S\wedge j\notin T$.
	
	Fix uniform sequences of c.e.\ sets $V_k$ and $U_k$ so that $i\in S$ if and only if there is some $k$ so that $V_k$ is infinite and $j\in T$ if and only if there is some $k$ so that $U_k$ is infinite. We construct two sequences of groups $(G_k)$ and $(H_k)$. We ensure that if $i\notin S$ then $\oplus_k G_k \oplus \oplus_k H_k$ is low. If $i\in S$ and $j\in T$ then there is some $k$ so that $H_k$ is universal. If $i\in S$ and $j\notin T$ then each $H_k$ is non-universal and some $G_k$ is $*$-universal.
	
	Let $(\oplus_k G_k\oplus \oplus_k H_k)_s$ be the group with the same  generators as $(\oplus_k G_k\oplus \oplus_k H_k)$, but only the relators enumerated by stage $s$. We enumerate relators declaring each group to be abelian at stage $0$. Thus the word problems of $(\oplus_k G_k\oplus \oplus_k H_k)_s$ are uniformly computable. 
	
	We have requirements as follows:
	
	$C_k$: If $V_k$ is infinite, then there is some $\ell$ so that $G_\ell$ is $*$-universal.
	
	$D_{\langle k,k'\rangle}$: If $V_k$ is infinite and $U_{k'}$ is infinite, then there is some $\ell$ so that $H_\ell$ is universal.
	
	$L_m$: If $\phi_{m,s}^{(\oplus_k G_k\oplus \oplus_k H_k)_s}(m)\downarrow$ for infinitely many $s$, then 
	$\phi_m^{\oplus_k G_k\oplus \oplus_k H_k}(m)\downarrow$.
	
	We order their priority by $C_0, L_0, D_0, C_1, L_1, D_1 \ldots$.
	
	To satsify $C_k$, when we see a new number enumerated into $V_k$, we act by:
	\begin{itemize}
		\item If $C_k$ is not yet initialized, we initialize it by choosing a new parameter $\ell$ so that the presentation of the group $G_\ell$ is not restrained by any higher priority $L_m$. 
		\item If $C_k$ is already initialized, we run one more step of the construction in Theorem \ref{thm:exists-sugs} to make $G_\ell$ be a $*$-universal group.
	\end{itemize}

	Observe that, regardless of outcome, every $G_k$ is either $*$-universal, or is finitely presented abelian so has computable word problem. In particular, no $G_k$ has universal word problem. 

	We act similarly for a $D_{\langle k,k'\rangle}$-requirement. Namely, when initialized, it chooses a new parameter $\ell$ and when we see new numbers enter $V_k$ and $U_{k'}$, we continue the coding to ensure that $H_\ell$ is a fixed abelian universal group. In the infinite outcome, $H_\ell$ is universal, and in the finite outcome, it has computable word problem.

	To satsify $L_m$, whenever $\phi_{m,s}^{(\oplus_k G_k\oplus \oplus_k H_k)_s}(m)\downarrow$, we place a restraint on the use of this computation, i.e., we place restraint against any new relators entering the presentations of a $G_k$ or $H_k$ used in this computation.
	
	Whenever we act, including placing restraint, all lower-priority requirements are reinitialized. The construction is put together as a standard finite injury argument. If $i\notin S$, then there is no $k$ so that $V_k$ is infinite, then every $L_m$-strategy eventually gets to succeed showing that $\oplus_k G_k\oplus \oplus_k H_k$ is low. If $i\in S$ and $j\in T$, then some $D$-requirement ensures that $\oplus_k G_k\oplus \oplus_k H_k$ is universal. Finally, if $i\in S$ and $j\notin T$, then one of the $G_\ell$ is $*$-universal. No $G_k$ has universal word problem, and since each $D$-requirement acts only finitely often, each of the $H_k$ have non-universal word problem. Thus $\oplus_k G_k\oplus \oplus_k H_k$ is $*$-universal by Corollary \ref{cor:oplusmaintains*universality}.
%
%
\end{proof}

\bibliographystyle{amsplain}
\bibliography{ISOG}

\end{document}